\documentclass[12pt]{amsart}

\usepackage{amscd}
\usepackage{amsmath}
\usepackage{amssymb}
\usepackage{amsthm}
\usepackage{epsf}
\usepackage{latexsym}
\usepackage{verbatim}
\usepackage[pdftex]{graphicx}
\usepackage{tikz}
\input epsf.tex

\usetikzlibrary{decorations.pathreplacing,calc}
\tikzset{thick/.style={line width=1pt}}

\newtheorem{thm}{Theorem}[section]

\newtheorem{cor}[thm]{Corollary}

\newtheorem{lem}[thm]{Lemma}

\theoremstyle{definition}

\numberwithin{equation}{section}

\newcommand{\Z}{\mathbb{Z}}
\newcommand{\N}{\mathbb{N}}
\newcommand{\R}{\mathbb{R}}

\newcommand{\Q}{\mathbb{Q}}

\newcommand{\sS}{\mathcal S}
\newcommand{\sV}{\mathcal V}
\newcommand{\sW}{\mathcal W}

\def\natural{\hbox{\rm\setbox1=\hbox{I}\copy1\kern-.45\wd1 N}}

\begin{document}


\baselineskip=17pt




\title[Infinite Transformations  ] {On Infinite Transformations with Maximal Control of Ergodic Two-fold Product Powers}

\author[T. M. Adams]{Terrence M. Adams}

\address{9161 Sterling Dr.\\ Department of Defense\\ Laurel, MD 20723}

\email{tmadam2@tycho.ncsc.mil}

\email{ }

\author[C. E. Silva]{Cesar E. Silva}

\address{Department of Mathematics and Statistics\\ Williams College\\ Williamstown, MA 01267} 

\email{csilva@williams.edu}

\subjclass[2010]{Primary 37A40; Secondary
37A05} 
\keywords{Infinite measure-preserving, ergodic, rationally ergodic, rank-one}


\begin{abstract}

We study the rich behavior of ergodicity and conservativity of 
Cartesian products of infinite measure preserving transformations.  
A class of transformations is constructed such 
that for any subset $R\subset \Q\cap (0,1)$ there 
exists $T$ in this class such that $T^p\times T^q$ is ergodic if and only 
if $\frac{p}{q} \in R$.  This contrasts with the finite measure 
preserving case where $T^p\times T^q$ is ergodic for all nonzero 
$p$ and $q$ if and only if $T\times T$ is ergodic.  
We also show that our class is rich in the behavior of conservative 
products.  

For each positive integer $k$, a family of rank-one infinite measure 
preserving transformations is 
constructed which have ergodic index $k$, but 
infinite conservative index.

\end{abstract}

\subjclass[2010]{Primary 37A25; Secondary 28D05}

\keywords{Ergodic index, Rank-One transformations}
\maketitle


\section{Introduction}

We consider measure spaces that are standard Borel spaces $(X,\sS,\mu)$ with a 
nonatomic $\sigma$-finite Borel measure $\mu$. A transformation is a
measurable map $T:X\to X$. We assume all our transformations are 
invertible mod $\mu$.  $T$ is measure preserving if $\mu(T^{-1}A)=\mu(A)$ for 
every measurable set $A$. The transformation $T$ is conservative if
for all sets of positive measure $A$ there exists a positive integer $n$ such that
$\mu(T^{-n}A\cap A)>0$, and it is ergodic if whenever a measurable set $A$ 
satisfies $T^{-1}A=A$ mod $\mu$, then $\mu(A)=0$ or $\mu (X\setminus A)=0$.
Since we assume our measures are nonatomic and the transformations are invertible,
ergodicity implies conservativity.

An infinite (or finite) measure-preserving transformation $T$ is weakly mixing
if for every ergodic finite measure-preserving transformation $S$, the 
transformation $T\times S$ is ergodic. When $T$ is finite measure-preserving,
if it is weakly mixing, then $T\times T$ is ergodic. As was shown by Aaronson, Lin, and Weiss 
\cite{AaLiWe79} this is no longer  the case in infinite measure. Furthermore, Kakutani and Parry \cite{KaPa63}  
had proved  earlier that there exist infinite measure-preserving transformations such that
$T\times T$ is ergodic but $T\times T\times T$ is not; such a transformation
is said to have {\bf ergodic index $2$}, and similarly one defines
{\bf ergodic index $k$}. They also constructed   infinite measure-preserving 
transformations $T$  where  all
finite Cartesian products of $T$  are ergodic (these transformations are said to have to have {\bf infinite ergodic index}).
A transformation $T$  is defined to be
{\bf power weakly mixing} if all its finite Cartesian products of nonzero powers  
are ergodic, \cite{DGMS99}.  Friedman and the authors proved in 
\cite{AdFrSi01} that there exist infinite ergodic index  (infinite measure-preserving) transformations
 such that $T\times T^2$ is not ergodic (in fact, not conservative), so in particular not
power weakly mixing.
It is now known that each of these notions is different
and that there are counterexamples in rank-one. In fact, there exist counterexamples
for actions of more general groups; the reader may refer to the surveys \cite{Da08}, \cite{DaSi09}
for actions of countable Abelian groups and to \cite{IaKaSi05}, \cite{DaPa11}  for flows.
More recently, the authors in \cite{JoSa}  investigate sets of conservative directions for $\Z^d (d>1),$  infinite measure-preserving actions, while we are only focused on integer actions. 

In this paper we continue the study of the unusual behavior of
ergodicity of Cartesian products in the case of  infinite measure-preserving
transformations.  All of our examples are rank-one and obtained by the 
technique of ``cutting and stacking.'' In the first part 
of the  paper, we define a   family of rank-one  infinite measure 
preserving transformations such that for each subset 
$R\subset \Q^1_0=\Q\cap (0,1)$, there exists a transformation $T$ in the family such that 
$T^p\times T^q$ is ergodic if and only if $\frac{p}{q} \in R$.  
($\Q$ represents the set of rationals in the real line.) 
We also show that given any subsets $R_1\subset R_2\subset \Q^1_0$, 
there exists $T$ in this family such that 
$T^p\times T^q$ is conservative ergodic for $\frac{p}{q} \in R_1$, 
$T^p\times T^q$ is conservative, but not ergodic for 
$\frac{p}{q}\in R_2\setminus R_1$, and $T^p\times T^q$ is not conservative 
for $\frac{p}{q}\in \Q\setminus R_2$. 

Milnor's notion of directional dynamics for $\Z^d$ actions can be applied 
in a similar manner to this setting.  See \cite{JoSa} for a complete definition. 
We say $T$ is $(r,2)$-conservative if for every $\varepsilon>0$, 
there exists $(p,q)\in \Z^2$ in the $\varepsilon$-strip around the line 
through the origin of slope $r$ such that $T^p \times T^q $ is conservative. 
The results of this paper show for any subset $R_1\subset \Q^1_0$ of rational directions, 
there is a rank-one infinite measure preserving transformation $T$ 
such that $T$ is $(\frac{p}{q},2)$-conservative, if and only if $\frac{p}{q}\in R_1$. 
In this paper, we do not consider the case of irrational directions.


  In the last section, we construct 
infinite measure-preserving transformations with ergodic index $k\in \N$, but 
with {\bf infinite conservative index }
(i.e., conservative $k$-fold Cartesian product for all $k\in\N$).  As in the paper 
by Kakutani and Parry \cite{KaPa63}, the ergodicity of the Cartesian 
products of the transformations constructed here can be characterized 
by the convergence or divergence of a special series.  
However our transformations have zero Krengel entropy and are of a different 
nature than those constructed in Kakutani and Parry \cite{KaPa63}  and 
\cite{AaLiWe79}, as those transformations are infinite Markov shifts that
are ergodic whenever they are conservative.

\section{Two-fold Products}

We construct our examples using the method of cutting and 
stacking.  We are able to restrict to the situation where 
we cut each column into four subcolumns of equal width and 
place a variable amount of spacers on the subcolumns.  
Since we make only four cuts at each stage, all of our 
transformations will be $\frac14$-rigid.  Thus, all Cartesian 
products of $T$ with itself will be conservative \cite[Corollary 1.4]{AdFrSi97}.
(A rank-one transformation whose two-fold Cartesian product is not conservative
was constructed in \cite{AdFrSi01}.)
When verifying that a transformation is conservative ergodic, 
we will show the equivalent 
property that for all sets $A$ and $B$  of positive measure, 
there is an integer $n>0$ such that $\mu(T^n A \cap B)>0$.

Before we prove our main result concerning this family, we will give 
three results which relate the placement of spacers to the 
ergodicity or conservativity of the Cartesian products.  
Since $T^p\times T^q$ is ergodic precisely when $T^q\times T^p$ is 
ergodic, the results extend to the set $\Q\cap (0,\infty)\setminus \{1\}$.

\subsection{Constructions}

We first recall the class of rank-one transformations constructed by  Friedman and the authors
in \cite{AdFrSi01}.  Let $a_n$, $b_n$, $c_n$ and $d_n$ be sequences of positive integers.  
We inductively define a sequence of columns $\{G_n\}$. Set $G_0$ to consist of the unit interval 
of height $H_0=1$. Assuming that   $G_n$ is a column of height $H_n$, form $G_{n+1}$ by 
cutting $G_n$ into four subcolumns of equal width and placing $a_n$, 
$b_n$, $c_n$ and $d_n$ spacers on the first, second, third and fourth 
subcolumns, respectively, and then stacking the subcolumns from left to right. Then 
this defines a Lebesgue measure-preserving transformation 
on a subset of $\R$.  Figure 1 represents $G_n$.  
Define $p_n=H_n+a_n$, $\ell_n=H_n+b_n$, 
$q_n=H_n+c_n$, $m_n=H_n+d_n$ and $h_{n+1}=p_n + \ell_n + q_n + H_n$ 
for $n\in \N$.  
Denote the transformation by $T=T_v$ where 
$v=\{(p_n,\ell_n,q_n,m_n): n\in \N \}$.  Let $\sW$ be the set of 
$v=\{(p_n,\ell_n,q_n,m_n): n\in \N \}$ satisfying 
$$\lim_{n\to \infty} \frac{p_n}{h_n} = \infty,$$
and $\ell_n$ and $m_n$ are chosen so that 
$$\ell_n > n\ (p_n + q_n + 2h_n )$$
and 
\[m_n > n\ (p_n+q_n+H_n+\ell_n)=n\  h_{n+1}.\]
Define $\sV=\{v\in \sW : p_n\leq q_n\}$. 
Much of our attention will focus on transformations $T_v$ 
such that $v\in \sV$. 
In this case, we still have freedom to choose the ratio $\frac{p_n}{q_n}$ as 
any rational in $(0,1)$.

\begin{figure}
    \def\hsq{4}
    \def\w{4}
    \def\hsm{1.3}

    \def\an{1}
    \def\bn{3}
    \def\cn{2}
    \def\dn{5}

  \centering

  \begin{tikzpicture}[scale=0.8]
    \draw (0,0) rectangle (\w,\hsq);
    \draw (0,\hsm) -- (\w,\hsm);
    \foreach \x in {1,...,\w} {
      \draw (\x,0) -- (\x,\hsq);
    }

    \draw[decorate,xshift=1.5mm,<->]
      (\w,0) -- node[midway,right] {$h_n$} (\w,\hsm);
    \draw[decorate,xshift=9mm,<->]
         (\w,0) -- node[midway,right] {$H_n$} (\w,\hsq);

    \draw (0,\hsq) rectangle node {$a_n$} ($(1,\hsq+\an)$);
    \draw (1,\hsq) rectangle node {$b_n$} ($(2,\hsq+\bn)$);
    \draw (2,\hsq) rectangle node {$c_n$} ($(3,\hsq+\cn)$);
    \draw (3,\hsq) rectangle node {$d_n$} ($(4,\hsq+\dn)$);
  \end{tikzpicture}

  \caption{Column $G_n$ \label{fig:} }
\end{figure}
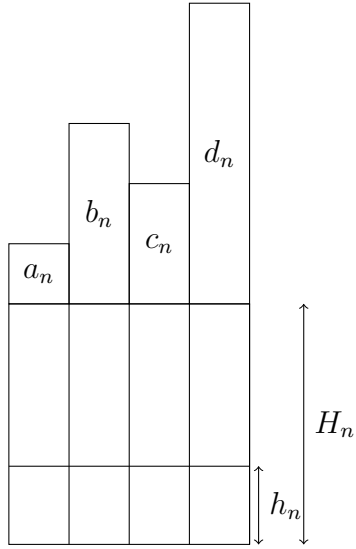



We recall the following theorem proved in \cite{AdFrSi01}.

\begin{thm}[\cite{AdFrSi01}] Let  $T=T_v$ with 
$v=\{(p_n,\ell_n,q_n,m_n):n\in \N\}\in \sW$.  Then $T$ has infinite conservative index.
If furthermore $v\in \sV$ and $\{n\in\N: |q_n- 2p_n|< 3 h_n\}$ is finite, then $T\times T^2$ is not
conservative and $T$ is not $2$-recurrent. If in addition $a_n=3h_n$ and $c_n= a_n+1$, then
$T$ has infinite ergodic index. 
\end{thm}

\subsection{Ergodicity and conservativity of two-fold products}

\begin{thm}
\label{erg}  
Let $p, q$ be natural numbers, and $T=T_v$ with 
$v=\{(p_n,\ell_n,q_n,m_n):n\in \N\}\in \sW$.  
If for each $k,\ell \in \N$, 
$$\{ n\in \N : \frac{q_n+\ell}{q} =\frac{p_n+k}{p}\in \N \}$$
is infinite, then $T^p\times T^q$ and $T^{-p}\times T^q$ 
are conservative ergodic.
\end{thm} 

\begin{proof} 
Let $E_1$ and $F_1$ each be subsets with positive measure  in the 
product space.  Choose levels $A,B,C$ and $D$ in $G_N$ for some $N\in\N$ 
such that 
\[
\mu \times \mu (E_1\cap (A\times B))> \frac{31}{32} \mu (A)\mu (B) 
\]
\[
\mu \times \mu (F_1\cap (C\times D))> \frac{31}{32} \mu (C)\mu (D). 
\]
Set $E=E_1\cap (A\times B)$ and $F=F_1\cap (C\times D)$.  
Let $k$ and $k^{\prime}$ be  integers such that 
$C=T^kA$ and $D=T^{k^{\prime}}B$.  We may assume both 
$k$ and $k^{\prime}$ are nonnegative by choosing a larger $N$ if needed.  
Choose $n\geq N$ such that 
\[
\frac{q_n+k^{\prime}}{q} =\frac{p_n+k}{p}=t_n
\]
is a positive integer. 
Let $A_n$ denote the part of $A$ which appears in the first subcolumn 
of $G_n$, and let $B_n$ denote the part of $B$ which appears in the 
third subcolumn of $G_n$.  
Then we have 
\[
T^{t_np}A_n = T^{p_n+k}A_n \subset C 
\]
and 
\[
T^{t_nq}B_n = T^{q_n+k^{\prime}}B_n \subset D. 
\]
Since $\mu (A_n)=\mu (B_n) =\frac14 \mu (A)$, then 
\[
\mu \times \mu ((T^p\times T^q)^{t_n}(A\times B) \cap (C\times D))\geq 
\frac{1}{16} \mu (C)\mu (D). 
\]
Therefore, 
\begin{align*}
\mu \times \mu ((T^p\times T^q)^{t_n}(E)\cap F)&\geq \mu \times \mu ((T^p\times T^q)^{t_n}(A\times B)\cap (C\times D))\\ 
&-\mu\times\mu(A\times B\setminus E)-\mu\times\mu(C\times D\setminus F)\\
&>\frac{1}{16} \mu (C)\mu (D)-\frac{1}{32} \mu (A)\mu (B) -\frac{1}{32} \mu (C)\mu (D)
\\&=0.
\end{align*}
showing that $T^p\times T^q$ is conservative ergodic. 

For the case of $T^{-p}\times T^q$, 
let $k$ and $k^{\prime}$ be positive integers such that $C=T^{-k}A$ and 
$D=T^{k^{\prime}}B$.  As in the previous case, we may assume both 
$k$ and $k^{\prime}$ are nonnegative by choosing a larger $N$ if needed. 
Let $A_n$ denote the part of $A$ which appears in the second subcolumn 
of $G_n$, and let $B_n$ denote the part of $B$ which appears in the 
third subcolumn of $G_n$.  
Then we have 
\[
T^{-t_np}A_n = T^{-p_n-k}A_n \subset C 
\]
and 
\[
T^{t_nq}B_n = T^{q_n+k^{\prime}}B_n \subset D. 
\]
A similar argument as before now shows that $T^{-p}\times T^q$ 
is conservative ergodic.
\end{proof} 

The previous proof may be applied to the special case where 
$F_1=E_1$, $C=A$, $D=B$ and $k=k^{\prime}$.  Thus we get the following 
sufficient condition implying $T^p\times T^q$ is conservative.

\begin{cor}
\label{cons}
Let $p, q$ be natural numbers, and $T=T_v$ with 
$v=\{(p_n,\ell_n,q_n,m_n):n\in \N\}\in \sW$.  If the set, 
$$\{ n\in \N : \frac{q_n}{q} =\frac{p_n}{p}\in \N \}$$
is infinite, then $T^p\times T^q$ and $T^{-p}\times T^q$ are conservative.
\end{cor}

\medskip

Now we consider the case when the products are not conservative.

\begin{thm}
\label{nonconserv}  
Let $p \neq q$ be natural numbers, and  let $T=T_v$ with 
$v=\{(p_n,\ell_n,q_n,m_n):n\in \N\}$ in $\sV$. 
 If $\{n\in \N : \vert pq_n - qp_n\vert \leq (p+q)h_n \}$  is finite, then  
$T^p\times T^q$ is  not conservative. 
\end{thm}

\begin{proof}  Since $T^p\times T^q$ is isomorphic to $T^q\times T^p$, it suffices
to prove the theorem assuming $p<q$.  Choose $N$ such that $N>q$, 
$\frac{N-1}{N}> \frac{p}{q}$  and 
$\vert pq_n - qp_n\vert > (p+q)h_n$ for all $n>N$. We may also assume
$p_n> 2 h_n$ for $n>N$.  Let $A$ be the bottom level of 
$G_{N+1}$.  Define 
\[\Lambda=
\{ i\in\N: \mu \times \mu ((T^p\times T^q)^i(A\times A)\cap (A\times A))>0 \}
\]
and for $n\in\N$, set 
\[\Lambda_n=\Lambda \cap (0,\frac1q H_n).
\]
We will prove inductively on $n$  that $\Lambda_n =\emptyset$.  This is 
true for $n=1,\ldots, N$ since $\mu (T^iA\cap A)=0$ for $1\leq i <H_N$.  

Suppose that our assertion holds for $\Lambda_n, n>N$.  We wish to verify that 
$\Lambda_{n+1}=\emptyset.$
Let  $D_n$ denote the union of the bottom $h_n$ levels in $G_n$.  
The set $A$ is  contained in $D_n$ for all $n>N$.  
Let $D_{n,r}$, for $r\in\{1,2,3 , 4\}$, denote  the part of $D_n$ in the $r^{th}$ subcolumn of $G_n$.  
Below we give a table showing the integers $j$,  $0 < j < 
\frac1q H_{n+1}$ such that $T^j D_{n,r_1}\cap D_{n,r_2}\neq\emptyset$.

\bigskip 

\mbox{  
\begin{tabular}{l|l|l|l|l}
 $j$     & $D_{n,1}$ & $D_{n,2}$ & $D_{n,3}$ & $D_{n,4}$ \\ 
\hline 
$T^jD_{n,1}$ & $0$ & $p_n-h_n$ &
$p_n + \ell_n - h_n$ &
$p_n + \ell_n + q_n - h_n$ \\ 
&   $h_n$  &   $p_n + h_n$ &
$p_n + \ell_n + h_n$ &
$p_n + \ell_n + q_n + h_n$\\ 
\hline 
$T^jD_{n,2}$ &   & $0$ &  $\ell_n - h_n$ &
$\ell_n + q_n - h_n$\\ 
&   &   $h_n$  &  $\ell_n + h_n$ &
$\ell_n + q_n + h_n$\\ 
\hline 
$T^jD_{n,3}$ &   &   & $0$ & $q_n - h_n$ \\ 
&   &   &   $h_n$ & $q_n + h_n$ \\ 
\hline 
$T^jD_{n,4}$ &   &   &   &  $0$ \\ 
&&&& $h_n$ 
\end{tabular}
}

\bigskip 

\noindent We will show there is no simultaneous intersection 
for $ip$ and $iq$ among the 10 intervals listed.  Let $K_{s,t}$ denote the interval in the $s^{\text th}$ row and $t^{\text th}$ column in the table;
for example $K_{2,3}$ is the interval $(\ell_n-h_n,\ell_n+h_n)$.  
To obtain the sets of intersection for $T^p$, divide 
the endpoints of the 
intervals in the table by $p$.  To obtain the sets of intersection 
for $T^q$, then divide the endpoints of the intervals by $q$.  

First note that 
\[
\frac{1}{q} H_n > \frac{1}{n-1} m_{n-1} > h_n.
\]
Therefore, by the inductive hypothesis, we may exclude 
 the interval $(0,h_n]$, since there cannot 
be simultaneous intersection on this interval with any other 
interval in the table.  

We next observe that  $\ell_n$ was chosen so that 
\[
\frac{p_n + h_n}{p} < \frac{\ell_n-h_n}{np} < \frac{\ell_n-h_n}{q} .
\]
This means that $\frac{1}{p} K_{1,2}<\frac{1}{q} K_{2,3}$. Therefore there is no positive integer  $i$
such that $ip\in K_{1,2}$ and $iq\in K_{2,3}$. As $p<q$, we also obtain
$\frac{1}{q} K_{1,2}<\frac{1}{p} K_{2,3}$, showing it cannot happen
that $iq\in K_{1,2}$ and $ip\in K_{2,3}$. Furthermore, we note that the left endpoint of each of the 
other intervals with $\ell_n$:  $K_{1,3}, K_{1,4}, K_{2,4}$ is to the right of the left endpoint
of $K_{2,3}$, showing similar inequalities for $K_{1,2}$ and these other intervals.
Similarly we have that   
\[
\frac{q_n + h_n}{p} < \frac{\ell_n-h_n}{np} < \frac{\ell_n-h_n}{q} .
\]
This means $\frac{1}{p}K_{3,4}<\frac{1}{q}K_{2,3}$, and by the same
argument as above it follows that there is no intersection of $K_{3,4}$ with 
 the 
other intervals with $\ell_n$.
 Thus,  it only remains to compare 
the intervals with $\ell_n$ as a summand 
in the endpoints to other intervals with $\ell_n$.  

Now, $\ell_n$ and $n$ are sufficiently large so that 
\begin{align*}
\frac{p_n+\ell_n+q_n+h_n}{q} &< 
\frac{(n-1)(p_n+\ell_n+q_n+h_n)}{np}\\
&<\frac{p_n+\ell_n+q_n+h_n - \frac{1}{n} \ell_n}{p}\\
&< \frac{\ell_n - h_n}{p},
\end{align*}
thus  $\frac{1}{q}K_{1,4}<\frac{1}{p}K_{2,3}$.
As we 
are comparing the worst case, this also shows  
\begin{align*}
\frac{1}{q}K_{1,3}<\frac{1}{p}K_{2,3}, &\text{ and }\frac{1}{q}K_{2,4}<\frac{1}{p}K_{2,3},\\
\frac{1}{q}K_{1,4}<\frac{1}{p}K_{1,3}, &\text{ and }\frac{1}{q}K_{1,4}<\frac{1}{p}K_{2,4}.
\end{align*}
Thus 
 there is no positive  integer  $i$
such that $iq$ is in $K_{1,4}$, 
$K_{1,3}$ or $K_{2,4}$, when $ip$ is in $K_{2,3}$, and such that 
$ip$ is in $ K_{1,3}$ or $K_{2,4}$ when $iq$ is in $K_{1,4}$.
Since $K_{2,3}< K_{1,4}$ and $p<q$ it cannot happen that
$iq\in K_{1,4}$ and $ip\in K_{2,3}$. Also, as
$q_n\geq p_n> 2h_n$ 
it follows that   $K_{1,3}<K_{1,4}, K_{2,4}<K_{1,4}$ and
$K_{2,3}<K_{1,3},K_{2,3}<K_{2,4}$, showing that the cases above
are all the ones we need to consider.

For the last part, we consider the case of simultaneous intersection 
on $K_{1,2}$ and $K_{3,4}$ (this argument will also cover the case of intersection on $K_{1,3}$ 
and $K_{2,4}$).

By the hypothesis,  
\[
\left\vert \frac{q_n}{q} - \frac{p_n}{p} \right\vert > 
\frac{h_n}{q} + \frac{h_n}{p} .
\]
Thus the distance between the centers of $\frac{1}{p}K_{1,2}$ and $\frac{1}{q}K_{3,4}$
is greater than the sum of their radii, showing there cannot be an intersection among these intervals.
Also, the  distance between the centers of $\frac{1}{q}K_{1,2}$ and $\frac{1}{p}K_{3,4}$
is
\[
\left\vert \frac{q_n}{p} - \frac{p_n}{q} \right\vert >\left\vert \frac{q_n}{q} - \frac{p_n}{p} \right\vert > 
\frac{h_n}{q} + \frac{h_n}{p},
\]also greater than the sum of their radii, so by there same argument there is no intersection in this case either. Finally, we note that the intervals
$K_{1,3}$ 
and $K_{2,4}$ have the same radii as $K_{1,2}$ and $K_{3,4}$, and the difference between
their centers is the same as the difference between the centers of $K_{1,2}$ and $K_{3,4}$,
thus the same argument as above applies to them, showing no intersection. 
   This completes the proof  under the assumption that $p<q$.  
   \end{proof}

   \noindent {\bf Remark:} (1). Many of the technical details involving the 
choice of $\ell_n$ and $m_n$ are not important.  Since one has 
freedom in choosing $\ell_n$ and $m_n$ independently of the choice 
of $p_n$ and $q_n$, then one may narrow the cases down to the case 
where block $D_{n,1}$ returns to block $D_{n,2}$ under $T^p$ and 
block $D_{n,3}$ returns to block $D_{n,4}$ under $T^q$.  In this case, 
the choice of $p_n$ and $q_n$ become important.  
(2). In all these constructions it is always the case that $T\times T^{-1}$ is conservative
as $\{a_n\}$ is a partial rigidity sequence for both $T$ and $T^{-1}$, so their product is conservative by \cite{AdFrSi97}.

   \begin{thm}
\label{nonerg}  
Let $p,  q$ be natural numbers, and $T=T_v$ with 
$v=\{(p_n,\ell_n,q_n,m_n):n\in \N\}$ in $\sV$. 
 If $\{n\in \N : \vert pq_n - qp_n\vert =0 \}$ is infinite and 
$\{n\in \N : 0\neq \vert pq_n - qp_n\vert \leq (p+q)h_n \}$ is finite, then  
 $T^p\times T^q$ is  
conservative and not ergodic. 
\end{thm}

\begin{proof} Conservativity of the product follows from Corollary~\ref{cons}.
Assume $p\leq q$. A similar argument to the one of Theorem~\ref{nonconserv} shows that $T^p\times T^q$ is 
not ergodic.  In this case, let 
\[\Lambda=
\{ i\in\N: \mu \times \mu ((T^p\times T^q)^i(A\times A)\cap (TA\times A))>0 \}.\]
Then it can be shown that $\Lambda =\emptyset$.  
The argument above shows that the only simultaneous return times 
that occur for the blocks $D_{n,r}$ under $T^p$ and $D_{n,s}$  under $T^q$ 
are when $r=s$.  But, when 
$p\leq q$, these blocks overlap perfectly; hence $\Lambda$ is empty.
\end{proof}

\subsection{Corollaries}
First we note that it follows from the proof of Theorems \ref{nonconserv} and \ref{nonerg}
that there is a set of positive measure $A$ such  that for all integers  $i\neq 0$, 
$ \mu \times \mu ((T^{pi}A\times T^{qi}A)\cap (A\times A)=0$. Therefore for all
integers  $i\neq 0$, $ \mu  (T^{pi}A\cap T^{qi}A\cap  A)=0$. In particular, $T$ 
is not multiply recurrent.

A corollary of Theorem \ref{nonerg} is that if $p\leq q$ and 
$\frac{p}{q}$ is not in the closure of 
$\{ \frac{p_n}{q_n} :n\in \N \}$, then 
$T^p\times T^q$ is not ergodic.  Without much effort 
this may be strengthened 
to say that if $\frac{p}{q}$ is not an accumulation  point 
of $\{ \frac{p_n}{q_n} :n\in \N \}$, then 
$T^p\times T^q$ is not ergodic.  
Let $F_N=\{ \frac{p_n}{q_n} : n\geq N\}$ and let 
$\bar{F_N}$ be the closure of $F_N$.  Let 
\[F=\bigcap_{N=1}^{\infty} \bar{F_N} \subset \bar{F_1}.\]  
Thus we have the following corollary.  

\begin{cor} 
If $p\leq q$ and $\frac{p}{q} \notin F$, then $T^p\times T^q$ is not ergodic.
\end{cor}

\noindent
{\bf Proof:} If $\frac{p}{q} \notin F$, then there exists $\epsilon >0$ 
such that $\{ n:\vert \frac{p_n}{q_n} -\frac{p}{q} \vert \leq \epsilon \}$ 
is finite.  Hence $\{ n:\vert qp_n -pq_n\vert \leq \epsilon qq_n\}$ 
is finite.  Since $\frac{q_n}{h_n} \to \infty$, then 
$\{ n:\vert qp_n -pq_n\vert \leq (p+q)h_n \}$ is finite.$\Box$

The condition in Theorem \ref{nonerg} actually says that 
$\frac{p}{q}$ must be well approximated by $\frac{p_n}{q_n}$, 
in a certain sense, for $T^p\times T^q$ to be ergodic.  
With some extra care we may construct sequences $p_n$ and $q_n$ so that 
each point $\frac{p}{q}$ in $(0,1)$ satisfies either the condition 
in Theorem \ref{erg} or the condition in Theorem \ref{nonerg}.  

\begin{cor} 
\label{ratcor} 
For each subset $R$ of $ \Q \cap (0,1)$, 
there exists an infinite 
measure-preserving transformation $T$ such that, for $ p, q\in\N$, 
where $p<q$, 
$T^p\times T^q$ is ergodic if and only if $\frac{p}{q} \in R$.  
\end{cor}
 
\begin{proof}
Order both $R$ and $S=(\Q\cap (0,1)) \setminus R$ as 
$R=\{ r_1,r_2,\cdots \}$ and $S=\{ s_1,s_2,\cdots \}$.  
Partition the natural numbers 
\[\N =\bigcup_{i=1}^{\infty} N(i)\] 
so that each $N(i)$ is infinite.  Let 
$k:\N \to \N \cup \{ 0\}$ 
and $\ell:\N \to \N \cup \{ 0\}$ be onto maps 
such that 
\[\{ j\in \N : k(j)=k,\ell(j)=\ell \}\ \ 
\hbox{is infinite for each pair}\ \ k,\ell \in 
\N.\]
Let 
\[N(i)=\{ n(i,1)<n(i,2)<\cdots \}.\]
For $i,j\in\N$, let $n_j=n(i,j)$, and write $r_i=\frac{p}{q}$ in reduced form.  
Let 
\[\delta_{i,j}=\min \{ \vert r_i-s_u\vert :1\leq u\leq i+j\}.\]
Choose $p_{n_j}$ and $q_{n_j}$ such that 
\[\frac{q_{n_j}+\ell(j)}{jq}=\frac{p_{n_j}+k(j)}{jp}\]
are both integers and such that 
$$q_{n_j}\delta_{i,j} >2h_{n_j} + k(j) + \ell(j).$$

By Theorem \ref{erg}, since $k(j)$ and $\ell(j)$ 
give all pairs of natural  
numbers for infinitely many $j$, then $T^p\times T^q$ is ergodic.  

Suppose now  $s_v=\frac{p}{q}\notin R$ for some $v$.  
Let 
\[\bar{n}=\max\{ n(i,j): 1\leq i+j\leq v\}.\]
For $n=n(i,j)>\bar{n}$, then 
\[q_{n_j}\delta_{i,j} >2h_{n_j} + k(j) + \ell(j)\]
where $\delta_{i,j}\leq \vert r_i-s_v\vert$.  Thus,
\begin{align*}
\vert p_n -q_n s_v\vert &\geq 
\vert q_n r_i - q_n s_v\vert - \vert p_n - q_n r_i\vert\\
&\geq q_n \delta_{i,j} - (k(j) + \ell(j))\\
&>2h_n\geq (s_v + 1)h_n.
\end{align*}
Since this is true for all $n>\bar{n}$, then by Theorem \ref{nonerg}, 
$T^p\times T^q$ is not ergodic.
\end{proof}

Let $p$ and $q$ be natural numbers and let $r=\frac{p}{q}$.  Define 
$$\Lambda=\{ n\in \N : \vert pq_n - qp_n\vert \leq (p+q)h_n \}$$
$$=\{ n\in \N : \vert rq_n - p_n\vert \leq (1+r)h_n \}$$
and let 
$$\Lambda_0=\{ n\in \N : rq_n - p_n = 0,
\ \ \frac{q_n}{q} \in \N \}.$$
Sequences $p_n$ and $q_n$ may be chosen so that 
$\Lambda_0 =\Lambda$ is infinite.  In this case, 
$T^p\times T^q$ will be conservative, but not ergodic.  Thus we obtain the 
following corollary whose proof is similar to Corollary \ref{ratcor}. 

\begin{cor} 
If 
$R_1\subset R_2\subset \Q \cap (0,1)$, 
then there exists an 
infinite Lebesgue measure-preserving transformation $T$ such that 
$T^p\times T^q$ is conservative ergodic for $\frac{p}{q} \in R_1$, 
$T^p\times T^q$ is conservative, but not ergodic for 
$\frac{p}{q}\in R_2\setminus R_1$, and $T^p\times T^q$ is not conservative 
for $\frac{p}{q}\in \Q_0^1 \setminus R_2$.  
\end{cor}

\section{Ergodic Index} 

We construct zero entropy transformations such that their ergodic index 
is determined by the convergence or divergence of a special series.  As in 
Kakutani and Parry \cite{KaPa63}, the ergodicity will depend on the divergence of 
a type of $p$-series where the terms being summed are 
explicit parameters in the construction of the transformations.  
However, all of the transformations constructed are partially 
rigid and hence have infinite conservative index.  The constructions in \cite{KaPa63}
are such that when the product is not ergodic it is not conservative.

\subsection{Constructions} 
We construct a family of rank-one transformations 
which we  classify by ergodic index.  
Let $L$ be a positive integer and let 
$$V_L = \{(u_1,u_2,\dots ,u_L):u_1<u_2<\dots < u_L\ \hbox{and }u_d\in{\N}\ 
\hbox{for } 1\leq d\leq L\} .$$  
We wish to define a sequence of vectors such that each vector 
$v$ in $V_L$ appears in our sequence along an arithmetic progression.  
First we list our countable set $V_L=\{v_j:j\in{\natural}\}$.  
Now we define our sequence $s:\natural\rightarrow V_L$ by
$$s(2^{j-1}+i2^j)=v_j$$
for $j=1,2,\dots$ and $i=0,1,\dots$.  
Finally we are ready to construct our rank-one transformations.

Let $r_n,n\in{\natural}$ be a nondecreasing sequence of positive integers 
satisfying $r_n>L$ for $n>0$.  Begin with column $C_1=[0,1)$.  
To obtain $C_{n+1}$: first denote $s(n)=(u_1,u_2,\dots ,u_L)$ 
and let $\sigma =u_1+\dots +u_L$.  
Cut column $C_n$ of height $h_n$ into $r_n$ subcolumns of equal width, 
and number the subcolumns from 1 to $r_n$, going from left to right. 
Place $h_n+u_d$ spacers on the $(r_n-L-1+d)$ subcolumn for $1\leq d\leq L$ 
and place $(2L+1)h_n+\sigma$ spacers on the $i^{\hbox{th}}$ 
subcolumn for $1\leq i\leq r_n-L-1$.  
Then stack each subcolumn onto the adjacent left subcolumn to 
form a column of height 
$g_n=r_nh_n+(r_n-L-1)((2L+1)h_n+\sigma )+(Lh_n+\sigma)$.  
Finally place $g_n$ spacers on top to form the column $C_{n+1}$ of height 
$h_{n+1}=2g_n$. Let $V_L^*$ be the family of all Lebesgue measure 
preserving transformations constructed in this manner. 

\subsection{Ergodic Index Characterization}

We give a criterion equivalent to the transformation 
$T \in V_L^*$ having an ergodic $k$-fold product.  
Let $X=\bigcup_{n=1}^{\infty} C_n$ and $\mu$ be Lebesgue measure 
on $X$.  Given a set $I\subset X$ of finite measure, 
let $\mu_I$ denote the measure $\mu$ conditioned on $I$. 
In particular, $\mu_I (J) = {\mu (J\cap I)} / {\mu (I)}$. 
Let $\mu^k$ denote the $k$-fold product measure on $X^k$. 
For a subset $A\subset X^k$ of finite measure, let 
$\mu_A^k$ be $\mu^k$ conditioned on $A$. 

\begin{thm} 
Let $k$ and $L$ be positive integers such that $1 < k \leq L$. 
A transformation $T \in V_L$ has an ergodic $k$-fold product 
if and only if 
$$\sum_{i=1}^{\infty}  \left({{1}\over{r_i}}\right)^k=\infty .$$
\end{thm}

\noindent 
If we let $r_n$ be a sequence of positive integers such that 
\[\sum_{i=1}^{\infty}  ({{1}\over{r_i}})^k=\infty\text{ but }
 \sum_{i=1}^{\infty}  ({{1}\over{r_i}})^{k+1}<\infty,\]  
then the theorem implies that $T$ has ergodic index $k$.  

\begin{cor} 
There exist rank-one transformations 
with ergodic index $k$ for any positive integer $k$.
\end{cor} 

\noindent {\bf Proof of ``if":} 
Now we set out to prove if 
$\sum_{i=1}^{\infty}  ({{1}\over{r_i}})^k=\infty$, 
then the $k$-fold product of $T$ is ergodic.  
First we state the Independence Lemma which follows directly 
from the construction.  The Independence Lemma coupled with the 
Borel-Cantelli Lemma imply that a product of levels sweeps out 
another product of levels.  Second we pair iterates to show that 
certain key iterates mix simultaneously with arbitrary measurable sets.  
This is accomplished in our Double Mixing Lemma.  
Finally we apply this in our Bumping Application 
to show that any set $E$ of positive measure ``bumps" any set 
$F$ of positive measure.

Given the column $C_n$, number the levels from top to bottom from 1 to $h_n$ respectively.  Now we are ready to state the Independence Lemma.

\begin{lem}
\label{indlem}
(Independence Lemma)  Let $n$ be a postive integer.  
Suppose $I$ and $J$ are levels in $C_n$.  
Let vector $v_j=(u_1,u_2,\dots ,u_L)$ satisfy $u_L<h_n$. 
Define the following sequence, 
$$t(i) = 2h_{\ell}$$
where 
$$\ell = \ell (i) = 2^{j-1}+(i+n)2^j .$$
Thus the sets $T^{-t(i)}J$, $i\in{\N}$ are 
independent with respect to $\mu_I$.  
Moreover, the sets $T^{t(i)}I$, $i\in{\N}$ 
are independent with respect to $\mu_J$.  
\end{lem}

\noindent The next lemma utilizes independence to obtain a single 
``mixing time''.  
Our technique is a spinoff of the Blum-Hansen method found in Friedman \cite{Fr83}.  
Then we prove the Double Mixing Lemma which is sufficient to give us the 
Bumping Application and complete 
our proof of ``if''.

\begin{lem} (Mixing Lemma) 
Let $(X,\gamma)$ be a probability space.  
Let $E_i\subset{X}$ be a sequence of pairwise independent sets satisfying 
$$\sum_{i=1}^{\infty}\gamma (E_i)=\infty .$$ 
Given any set $E\subset X$ and $\varepsilon >0$, there exists a 
positive integer $i$ such that 
$\gamma (E\cap E_i)>(\gamma (E)-\varepsilon )\gamma (E_i)$. 
\end{lem}

\noindent {\bf Proof:}  By squaring the integrand and applying independence, 
we get the following, 
$$\int (\frac1{N} \sum_{i=1}^{N}
({\mathcal X}_{E_i}-\gamma (E_i)))^2 d\gamma 
= \frac1{N^2} \sum_{i=1}^{N} \gamma (E_i) (1 - \gamma (E_i)) 
< \frac1{N^2} \sum_{i=1}^{N} \gamma (E_i) .$$
The Cauchy-Schwartz inequality implies 
\begin{align*}
| \frac1{N} \sum_{i=1}^{N} (\gamma (E\cap E_i)-\gamma (E)\gamma (E_i)) |
&= | \int_{E} (\frac1{N} \sum_{i=1}^{N}({\mathcal X}_{E_i}-\gamma (E_i)))d\gamma | \\ 
&< \frac1N \sqrt{\sum_{i=1}^{N} \gamma (E_i)}.
\end{align*}

Thus for $\varepsilon>0$,
$$\frac{\sum_{i=1}^{N} 
(\gamma (E\cap E_i)-\gamma (E)\gamma (E_i))}{\sum_{i=1}^{N} 
\gamma (E_i)}<
\frac{\sqrt{\sum_{i=1}^{N} \gamma (E_i)}}{\sum_{i=1}^{N} 
\gamma (E_i)} \to 0$$
as $N\to \infty$, since $\sum_{i=1}^{\infty} 
\gamma (E_i)=\infty.\ \ \ \ \Box$

In our examples, $\gamma=\mu^k_A$ and 
$E_i=A\cap T_k^{-t(i)}B$ where $A$ and $B$ are products of levels, 
and $T_k$ is the $k$-fold Cartesian product of $T$. 
The following lemma applies to $\mu^k_A$, $A\cap T_k^{-t(i)}B$, 
and $\nu=\mu^k_B$ and $F_i=T_k^{t(i)}A\cap B$.

\begin{lem} (Double Mixing Lemma) 
Let $(X,\gamma)$ and $(Y,\nu)$ be probability spaces.  
Suppose $(E_i)_{i=1}^{\infty}\subset{X}$ and $(F_i)_{i=1}^{\infty}
\subset{Y}$ 
are each a sequence of pairwise independent sets such that 
$\gamma (E_i)=\nu (F_i)$ 
and $\sum_{i=1}^{\infty}\gamma (E_i)=\sum_{i=1}^{\infty}\nu (F_i)=\infty$.  
Then given sets $E\subset X$ and $F\subset Y$ and $\varepsilon>0$, 
there exists a positive integer $i$ such that both $\gamma (E\cap E_i)>
(\gamma (E)-\varepsilon )\gamma (E_i)$ 
and 
$\nu (F\cap F_i)>(\nu (F)-\varepsilon )\nu (F_i)$.
\end{lem}

\noindent {\bf Proof:}  
Let $\Lambda=\{i:\gamma (E\cap E_i)>(\gamma (E)-\varepsilon)\gamma (E_i)\}$.  
The Mixing Lemma implies that 
$$\sum_{i\in \Lambda} \nu (F_i)=\sum_{i\in \Lambda} \gamma (E_i)=\infty.$$
Therefore by applying the Mixing Lemma once again we get that 
there exists $i\in \Lambda$ such that 
$\nu (F\cap F_i)>(\nu (F)-\varepsilon )\nu (F_i).\ \ \ \ \Box$

\smallskip

\noindent {\bf Bumping Application}:  The transformation $T_k$ is ergodic.
\vskip .1in
\noindent {\bf Proof}:  
Let $E$ and $F$ be sets of positive measure in the $k$-fold product space.  
Choose a positive integer $n$ and levels 
$A_1,\dots ,A_k,B_1,\dots ,B_k$ in $C_n$ so that \[\mu^k 
(E\cap (A_1\times \dots \times A_k))>{3\over 4}\prod_{i=1}^k
\mu (A_i)\]
and 
\[\mu^k (F\cap (B_1\times \dots \times B_k))
>{3\over 4}\prod_{i=1}^k\mu (B_i).\]
By choosing a larger $n$ if necessary, we may assume $A_i$ appears 
above $B_i$ in $C_n$. 
Choose the vector 
$v_j=(u_1,u_2,\dots , u_L)$ satisfying both $u_L<h_n$ and 
for each $p\in{\N}$ there exists $d\in{\N}$ such that 
$u_d=$(position of $A_p)-($position of $B_p$).  
As in Lemma \ref{indlem}, let $t(i)=2h_{\ell}$ where 
$\ell = \ell (i) = 2^{j-1}+(i+n)2^j$. 
The Independence Lemma implies for each $p$, 
$A_p\cap T_k^{-t(i)}B_p$, $i\in{\natural}$ are independent with respect to 
$\mu_{A_p}$.  Thus
$E_i=A\cap T_k^{-t(i)}B$, $i\in{\natural}$ are independent 
with respect to $\mu^k_A$.  Similarly $F_i=B\cap T_k^{t(i)}A$, 
$i\in{\natural}$ are independent with respect to $\mu^k_B$.   
Since 
$$\sum_{i=1}^{\infty}( {1} / {r_{\ell (i)}})^k =\infty ,$$ 
then 
$$\sum_{i=1}^{\infty}\mu_A^k (E_i)=\sum_{i=1}^{\infty}\mu_B^k (F_i)=\infty.$$ 
Therefore by the Double Mixing Lemma with $\gamma=\mu^k_A$, 
$\nu=\mu^k_B$ and $\varepsilon=\frac14$, there exists a positive integer 
$i$ such that both $\mu^k_A (E\cap E_i)>\frac12 \mu^k_A (E_i)$ 
and $\mu^k_B (F\cap F_i)>\frac12 \mu^k_B (F_i)$.  Hence 
$\mu^k(T^{t(i)}E\cap F)>0.\ \ \ \ \Box$

\vskip .6in

\noindent {\bf Proof of ``only if'':}  
Choose $n\geq 2$ so that 
\[\sum_{i=n}^{\infty}\left({{L+1}\over{r_i}}\right)^k<1.\]
For $1\leq i\leq h_n$ let $I_i$ denote the $i^{th}$ level of $C_n$ 
from top to bottom.  For $m\geq n$ and $1\leq j\leq r_m$, 
let $C_{m,j}$ be th $j^{th}$ subcolumn of $C_m$.  Denote 
$${\mathcal R}_m=\bigcup_{i=r_n-L}^{r_n}C_{m,i}\ \ \hbox{and}\ \ {\mathcal L}_m
=C_m\setminus {\mathcal R}_m.$$
Define $A=I_1\times \dots \times I_1$ and 
$B=[I_1\times \dots I_1\times I_2]\setminus [\bigcup_{m=n}^{\infty}
({\mathcal R}_m\times \dots \times {\mathcal R}_m).$
Thus 
$$\mu^k(B)=[\prod_{m=n}^{\infty}(1-({L\over{r_m}})^k)]\mu(I_1)^k
\geq (1-\sum_{m=n}^{\infty}({L\over{r_m}})^k))\mu(I_1)^k>0.$$
We will prove inductively on $m$ that 
$$\mu^k(T_k^iA\cap B)=0$$
for $-h_m\leq i\leq h_m$.  Since $A$ and $B$ are disjoint and 
each is a product of levels from $C_n$, then $\mu^k(T_k^iA\cap B)=0$ 
for $-h_n\leq i\leq h_n$.  Now suppose $\mu^k(T_k^iA\cap B)=0$ for 
$-h_m\leq i\leq h_m$.  

Consider the four intersections:  
$T^i{\mathcal L}_m\cap {\mathcal L}_m$, $T^i{\mathcal L}_m\cap {\mathcal R}_m$, 
$T^i{\mathcal R}_m\cap {\mathcal L}_m$, and $T^i{\mathcal R}_m\cap {\mathcal R}_m$.  

First $\mu (T^i{\mathcal R}_m\cap {\mathcal L}_m)=0$ for $h_m\leq i\leq h_{m+1}$.

Second if $\mu (T^i{\mathcal R}_m\cap {\mathcal R}_m)>0$ 
then $\mu (T^i{\mathcal L}_m\cap C_m)=0$ for $h_m\leq i\leq h_{m+1}$.  
Since $B\cap({\mathcal R}_m\times \dots \times {\mathcal R}_m)=\emptyset$, 
we need only consider the intersections:  $T^i{\mathcal L}_m\cap {\mathcal L}_m$ and 
$T^i{\mathcal L}_m\cap {\mathcal R}_m$.

Suppose $\mu (T^i{\mathcal L}_m\cap {\mathcal R}_m)>0$.  
If $\mu (T^i{\mathcal L}_m\cap ({\mathcal R}_m\setminus C_{m,r_n-L+1}))>0$ 
then $\mu (T^i{\mathcal L}_m\cap {\mathcal L}_m)=0$.  
Since $B\cap ({\mathcal R}_m \times \dots \times {\mathcal R}_m)=\emptyset$, 
we may assume $\mu (T^i{\mathcal L}_m\
\cap ({\mathcal R}_m \setminus C_{m,r_n-L+1}))=0$.  

Finally consider the case 
$\mu (T^i{\mathcal L}_m\cap ({\mathcal L}_m\cup C_{m,r_n-L+1}))>0$.  
For $j\leq r_n-L$ and for $p\geq j$, $T^iC_{m,j-1}$ overlaps 
the same number of levels in $C_{m,p-1}$ as $T^iC_{m,j}$ 
overlaps in $C_{m,p}$.  Hence this scenario reduces 
to the case $-h_m\leq i\leq h_m$.  
Therefore $\mu (T_k^iA\cap B)=0$ for $0\leq i \leq h_{m+1}$.  
The case $-h_{m+1}\leq i $ can be handled in a similar manner.  
$\Box$

\subsection*{Acknowledgements}
We would like to thank Nat Friedman for 
  discussions at the early stages of  this work.   We thank James Wilcox
  for typesetting  our figure in tikz.

\normalsize

\bibliographystyle{amsalpha}
\bibliography{MixingRankOneBib}

\end{document}